\documentclass[twoside,leqno]{amsart}
\usepackage{amssymb,enumerate}
\usepackage{amsmath}
\usepackage{graphics}
\usepackage{graphicx}

\usepackage{amsfonts}
\usepackage{amssymb}
\newtheorem{thm}{Theorem}[section]
\newtheorem{lem}{Lemma}[section]
\newtheorem{cor}[thm]{Corollary}
\newtheorem{rem}[thm]{Remark}

\newtheorem{defin}{Definition}[section]
\newtheorem{pr}{Proposition}[section]

\date{}

\begin{document}

\begin{center}
\Large{\bf Weighted Ostrowski-Gr\"{u}ss type inequalities
}\vspace{0.5cm}

 \Large{Ana Maria Acu, Heiner Gonska}
\end{center}\vspace{1.1cm}

\thispagestyle{plain}



\noindent{\bf Abstract.} {\small Several inequalities of
Ostrowski-Gr\"{u}ss-type availabe in the literature are generalized
consi-\linebreak dering the weighted case of them. Involving the least concave
majorant of the modulus of continuity we provide  upper bounds of
our inequalities.}


\begin{center}

\end{center}
\vspace*{0.3cm}

\section{Introduction}
Over the last decades,  integral inequalities have attracted much
attention because of their applications in statistical analysis and
the theory of distributions. In this paper we improve the classical
Ostrowski type inequality for weighted integrals and generalize some
Ostrowski-Gr\"{u}ss type inequalities involving differentiable
mappings. Also, applications to special weight functions are
investigated.

 For each $x\in[a,b]$ consider the linear
functional
 $${\mathcal{L}}(f)(x):=\displaystyle{f(x)-\frac{1}{b-a}\int_{a}^bf(t)dt-\frac{f(b)-f(a)}{b-a}\left(x-\frac{a+b}{2}\right)}, f\in C[a,b].$$

 If $f:[a,b]\to \mathbb{R}$  is  differentiable with bounded derivative, then
\begin{eqnarray}
\left|{ \mathcal{L}}(f)(x)\right|&\leq&
\displaystyle\frac{1}{8}(b-a)(\Gamma-\gamma)\label{ec1.1}\\
&\leq&
\displaystyle\frac{1}{4\sqrt{3}}(b-a)(\Gamma-\gamma)\label{ec1.2}\\
&\leq& \displaystyle\frac{1}{4}(b-a)(\Gamma-\gamma),\label{ec1.3}
\end{eqnarray}
where  $\gamma:=\inf\{ f^{\prime}(x)| x\in[a,b]\}$ and  $\Gamma:=\sup\{ f^{\prime}(x)| x\in[a,b]\}$ .

 The inequality (\ref{ec1.3}) was proven by
S.S. Dragomir and S. Wang \cite{5} and it is known as the
Ostrowski-Gr\"{u}ss-type inequality. This inequality was
improved by M. Mati\'{c} et al. \cite{10}, and we recall their
result in (\ref{ec1.2}). An improvement of this result was given by
X.L. Cheng in \cite{4}, as shown in relation (\ref{ec1.1}). He also
proved that the constant $1/8$ is best possible.

In \cite{GRR} the authors introduced the linear functional
${\mathcal{L}}_{c}:C[a,b]\to \mathbb{R}$ defined by
 $${\mathcal{L}}_c(f)(x):=\displaystyle{f(x)-\frac{1}{b-a}\int_{a}^bf(t)dt-c\cdot\frac{f(b)-f(a)}{b-a}\left(x-\frac{a+b}{2}\right)},\textrm{ where } c\geq 0.$$
 The following result gives  bounds of the functional
 ${\mathcal{L}}_c$ involving differences of upper and lower bounds
 of first order derivatives.
 \begin{thm} \cite{GRR} For all $x\in [a,b]$, $c\in[0,2]$ and $f\in
 C^1[a,b]$ we have
 \begin{equation}\label{e33}
 \displaystyle\frac{1}{2(b-a)}\left[(x-a-u_{c}(x))^2\gamma-(x-b-u_{c}(x))^2\Gamma\right]\leq {\mathcal{L}}_c(f)(x)
 \end{equation}
 $$\leq\displaystyle\frac{1}{2(b-a)}\left[(x-a-u_{c}(x))^2\Gamma-(x-b-u_{c}(x))^2\gamma\right], $$
 where $ u_{c}(x):=c\left(x-\displaystyle\frac{a+b}{2}\right)$.
 \end{thm}
 \begin{rem} a) From (\ref{e33}), with $c=1$,
 inequality (\ref{ec1.1}) follows which was established  by
X.L. Cheng in \cite{4}.

b) As a consequence of (\ref{e33}), for $c=0$, the following
inequality holds:
$$ \displaystyle -\frac{(x-a)^2+(b-x)^2}{2(b-a)}\| f^{\prime}\|_{\infty}\leq \frac{(x-a)^2\gamma-(b-x)^2\Gamma}{2(b-a)}
\leq f(x)-\frac{1}{b-a}\int_{a}^b f(t)dt $$
$$ \leq \displaystyle\frac{(x-a)^2\Gamma-(b-x)^2\gamma}{2(b-a)}\leq\frac{(x-a)^2+(b-x)^2}{2(b-a)}\| f^{\prime}\|_{\infty}. $$
This inequality improves the classical Ostrowski inequality
presented by Anastassiou in \cite{3} in the form
$$ \left| f(x)-\frac{1}{b-a}\displaystyle\int_{a}^b f(t)dt\right|\leq \displaystyle\frac{(x-a)^2+(b-x)^2}{2(b-a)}\| f^{\prime}\|_{\infty}. $$
 \end{rem}
 Weighted versions of (\ref{ec1.3}), (\ref{ec1.2}) and
 (\ref{ec1.1}) were established by J. Roumeliotis  in \cite{7} and
 \cite{R}. These results are given below.

 \begin{defin}Let $w:(a,b)\to (0,\infty)$ be integrable, i.e.,
$\displaystyle\int_{a}^{b}w(t)dt<\infty$, then
$m(\alpha,\beta):=\displaystyle\int_{\alpha}^{\beta}w(t)dt$ and  $
M(\alpha,\beta):=\displaystyle\int_{\alpha}^{\beta}tw(t)dt  $
are the first moments, for $[\alpha,\beta]\subseteq [a,b]$. Define
the mean of the interval $[\alpha,\beta]$ with respect to the weight
function $w$ as
$\sigma(\alpha,\beta):=\displaystyle{\frac{M(\alpha,\beta)}{m(\alpha,\beta)}}
$.
\end{defin}
The weighted variant of the functional ${ \mathcal{L}}$ can be
written in the following way:
$$ \mathcal{L}_{w}(f)(x):=f(x)-\displaystyle\frac{1}{m(a,b)}\int_{a}^bf(t)w(t)dt-\displaystyle\frac{f(b)-f(a)}{b-a}\left(x-\sigma(a,b)\right).$$
\begin{thm}\cite{7} Let $f\in[a,b]\to \mathbb{R}$ be
differentiable with bounded derivative and let\linebreak  $w:(a,b)\to (0,\infty)$ be
integrable. Then
\begin{align*}\left|\mathcal{L}_{w}(f)(x) \right|&\leq
\displaystyle\frac{1}{2}(\Gamma-\gamma)\frac{\sqrt{b-a}}{m(a,b)}\left\{\int_a^b
K^2(x,t)dt-\frac{m(a,b)^2(x-\sigma(a,b))^2}{b-a}\right\}^{\frac{1}{2}}\\
&\leq \displaystyle\frac{1}{4}(\Gamma-\gamma)(b-a),
\end{align*}
for all $x\in[a,b]$, where $K(x,t)=\left\{\begin{array}{l}
\int_{a}^tw(u)du, a\leq t\leq x\\
\vspace{-0.4cm}\\
 \int_b^tw(u)du, x<t\leq b.\end{array}\right.$
\end{thm}
\begin{thm}\cite{R} Let $f\in[a,b]\to \mathbb{R}$ be
differentiable with bounded derivative and let\linebreak  $w:(a,b)\to (0,\infty)$ be
integrable. Then
\begin{equation}\label{e34}
\left| \mathcal{L}_{w}(f)(x) \right|\leq
\displaystyle\frac{\Gamma-\gamma}{m(a,b)}\int_{x}^{t^*}(t-x)w(t)dt,
\end{equation}
 for all $x\in[a,b]$, where $t^*\in[a,b]$ is unique and verifies
 $$ \displaystyle\frac{m(a,b)}{b-a}|x-\sigma(a,b)|=\left\{\begin{array}{l}m(t^*,b), a\leq x\leq \sigma(a,b)\\
 \vspace{-0.4cm}\\
 m(a,t^*),\sigma(a,b)<x\leq b. \end{array}\right. $$
\end{thm}

For each $x\in[a,b]$ consider the linear functional
$\mathcal{L}_{w,c}:C[a,b]\to \mathbb{R}$ defined by
$$ \mathcal{L}_{w,c}(f)(x):=f(x)-\displaystyle\frac{1}{m(a,b)}\int_{a}^bf(t)w(t)dt-c\cdot\displaystyle\frac{f(b)-f(a)}{b-a}\left(x-\sigma(a,b)\right),
\textrm{ where } c\geq 0.$$ In this paper we propose the weighted
analogue of (\ref{e33}). Also, new inequalities of $
\mathcal{L}_{w,c}$ will be considered involving the least concave
majorants of the first order moduli of continuity.

\section{Generalized Ostrowski-Gr\"{u}ss type inequalities}
In this section we will give the upper bounds of $\mathcal{L}_{w,c}$
involving differences of upper and lower bounds of first order
derivatives. First, we need the following  lemma.

 \begin{lem} For $c\in [0,1], $ there exists a unique $t^*\in[a,b]$
 satisfying
 $$ u_{w,c}(x)=\left\{\begin{array}{l}
 \displaystyle\frac{1}{m(a,b)}\int_{b}^{t^*}w(u)du,\, a\leq x\leq \sigma(a,b)\\
 \vspace{-0.4cm}\\
 \displaystyle\frac{1}{m(a,b)}\int_{a}^{t^*}w(u)du,\,
 \sigma(a,b)<x\leq b,
 \end{array}\right. $$
 where $u_{w,c}(x)=\displaystyle\frac{c}{b-a}\left(x-\sigma(a,b)\right)$.
 \end{lem}
 \begin{proof} Let us consider $a\leq x\leq \sigma(a,b)$ and
 $f(t)=\displaystyle\frac{1}{m(a,b)}\int_{b}^tw(u)du-u_{w,c}(x)$, $t\in
 [x,b]$.\\
 Since $f$ is strictly increasing on $(x,b)$, $f(b)=-u_{w,c}(x)\geq
 0$ and
 $f(x)=\displaystyle\frac{1}{m(a,b)}\int_{b}^xw(u)du-u_{w,c}(x)$,
 then to show that $t^*\in[x,b]$ exists it will suffice to establish
 that $f(x)\leq 0$.\\
 It follows
 \begin{align*}
 u_{w,c}(x)&=\displaystyle\frac{c}{b-a}\left(x-\sigma(a,b)\right)=\frac{c}{b-a}\left(x-\frac{M(a,b)}{m(a,b)}\right)=
 -\frac{c}{(b-a)m(a,b)}\int_a^b(t-x)w(t)dt\\
 &=-\displaystyle\frac{c}{(b-a)m(a,b)}\left\{\int_a^x(t-x)w(t)dt+\int_x^b(t-x)w(t)dt\right\}\\
 &\geq -\displaystyle\frac{c}{(b-a)m(a,b)}\int_x^b(t-x)w(t)dt\geq -\displaystyle\frac{c}{(b-a)m(a,b)}(b-x)\int_x^b
 w(t)dt\\
 &\geq -\displaystyle\frac{c}{m(a,b)}\int_x^bw(t)dt\geq \displaystyle\frac{1}{m(a,b)}\int_b^x
 w(t)dt.
 \end{align*}In a similar way for $\sigma(a,b)<x\leq b$ follows that there exists a unique $t^*\in
 [a,x]$ such that $u_{w,c}(x)=\displaystyle\frac{1}{m(a,b)}\int_{a}^{t^*}w(u)du$.
 \end{proof}

 \noindent Denote
$$ {\mathcal P}(x,t)=\left\{\begin{array}{l}\displaystyle\frac{1}{m(a,b)}\int_{a}^t w(u)du-u_{w,c}(x), a\leq t<x,\\
\vspace{-0.4cm}\\
\displaystyle\frac{1}{m(a,b)}\int_{b}^t w(u)du-u_{w,c}(x), x\leq
t\leq b.\end{array}\right. $$
 It is easy to verify that, for all $f\in C^{1}[a,b]$,
 $ \mathcal{L}_{w,c}(f)(x)=\displaystyle\int_a^b{\mathcal{P}(x,t)f^{\prime}(t)dt}$.

 \begin{thm}\label{t2.1} For all $x\in[a,b], c\in[0,1]$ and $f\in C^1[a,b]$ we
 have
 \begin{equation}\label{e3.3}(1-c)(x-\sigma(a,b))\gamma+(\gamma-\Gamma)\nu(x,t^*)\leq
  \mathcal{L}_{w,c}(f)(x) \leq
  (1-c)(x-\sigma(a,b))\Gamma+(\Gamma-\gamma)\nu(x,t^*),\end{equation}
 where
 $\nu(x,t^*):=\displaystyle\frac{1}{m(a,b)}\int_{x}^{t^*}(t-x)w(t)dt$.
 \end{thm}
 \begin{proof}
 If $a\leq x\leq \sigma(a,b)$, then
 $${\mathcal{P}}(x,t)\geq 0,\textrm{ for } t\in[a,x]\cup[t^*,b] \textrm{ and } { \mathcal{P}}(x,t)<0, \textrm{ for } t\in(x,t^*). $$
 Also, if $\sigma(a,b)\leq x\leq b$, it follows
$${\mathcal{P}}(x,t)\leq 0,\textrm{ for } t\in[a,t^*]\cup[x,b] \textrm{ and } { \mathcal{P}}(x,t)>0, \textrm{ for } t\in(t^*,x). $$
Let $a\leq x\leq \sigma(a,b)$. It follows
$$ \mathcal{L}_{w,c}(f)(x)\leq \Gamma\left(\displaystyle \int_{a}^x{\mathcal P}(x,t)dt+\int_{t^*}^b{\mathcal P}(x,t)dt \right)+\gamma\int_{x}^{t^*}{\mathcal P}(x,t)dt.$$
We have
 \begin{align*} \displaystyle \int_{a}^x\!\!{\mathcal
P}(x,t)dt\!+\!\int_{t^*}^b\!\!{\mathcal P}(x,t)dt\!=\!&\displaystyle
\int_{a}^x\left(\displaystyle\frac{1}{m(a,b)}\int_{a}^t\!\!w(u)du\!-\!u_{w,c}(x)\right)dt\!+\!\int_{t^*}^b\!\!\left(\displaystyle\frac{1}{m(a,b)}\int_{b}^t\!\!w(u)du\!-\!u_{w,c}(x)\right)dt\\
&=\displaystyle\frac{1}{m(a,b)}\left\{x\int_{a}^xw(u)du-\int_{a}^xtw(t)dt-t^*\int_{b}^{t^*}w(u)du-\int_{t^*}^btw(t)dt\right\}\\
&-u_{w,c}(x)(x\!-\!a\!+\!b\!-\!t^*)\!=\!\displaystyle\frac{1}{m(a,b)}\left\{x\int_{a}^xw(u)du\!-\!\int_{a}^x\!\!tw(t)dt-\int_{t^*}^btw(t)dt\right\}\\
&-u_{w,c}(x)(x-a+b)+t^*\left(u_{w,c}(x)-\displaystyle\frac{1}{m(a,b)}\int_{b}^{t^*}w(u)du\right)\\
&=\displaystyle\frac{1}{m(a,b)}\left\{x\int_{a}^x\!\!w(u)du\!-\!\int_a^x\!\!
tw(t)dt\!-\!(x\!-\!a\!+\!b)\int_b^{t^*}\!\!w(u)du\!-\!\int_{t^*}^b\!\!tw(t)dt\right\}\\
&=\displaystyle\frac{1}{m(a,b)}\left\{\int_{t^*}^b(x-a+b-t)w(t)dt+\int_a^x(x-t)w(t)dt\right\}\\
&=\displaystyle\frac{1}{m(a,b)}\left\{(b-a)\int_{t^*}^b
w(t)dt+\int_a^b(x-t)w(t)dt-\int_x^{t^*}(x-t)w(t)dt\right\}\\
&=\displaystyle\frac{1}{m(a,b)}\left\{-(b\!-\!a)m(a,b)u_{w,c}(x)\!+\!xm(a,b)\!-\!M(a,b)\!+\!\int_{x}^{t^*}\!\!(t\!-\!x)w(t)dt\right\}\\
&=\displaystyle(1-c)(x-\sigma(a,b))+\frac{1}{m(a,b)}\int_{x}^{t^*}(t-x)w(t)dt.
\end{align*}
\begin{align*}
\displaystyle{\int_{x}^{t^*}{\mathcal
P}(x,y)dt}&=\displaystyle\int_{x}^{t^*}\left\{\frac{1}{m(a,b)}\int_{b}^tw(u)du-u_{w,c}(x)\right\}dt\\
&=\displaystyle\frac{1}{m(a,b)}\left\{t^*\int_{b}^{t^*}w(u)du-x\int_{b}^xw(u)du-\int_{x}^{t^*}tw(t)dt\right\}-u_{w,c}(x)(t^*-x)\\
&=\displaystyle
t^*\left\{\frac{1}{m(a,b)}\int_{b}^{t^*}w(u)du-u_{w,c}(x)\right\}-\frac{1}{m(a,b)}\left\{x\int_{b}^x
w(u)du+\int_{x}^{t^*}tw(t)dt\right\}\\
&+x\cdot\displaystyle\frac{1}{m(a,b)}\int_{b}^{t^*}w(u)du=-\frac{1}{m(a,b)}\int_{x}^{t^*}(t-x)w(t)dt=-\nu(x,t^*).
\end{align*}
Therefore,
\begin{align}
\mathcal{L}_{w,c}(f)(x)&\leq
\Gamma\left[(1-c)(x-\sigma(a,b))+\nu(x,t^*)
\right]-\gamma\nu(x,t^*)\label{e3.2}\\
&=(1-c)(x-\sigma(a,b))\Gamma+(\Gamma-\gamma)\nu(x,t^*).\nonumber
\end{align}
By  similar reasoning it follows that (\ref{e3.2}) is also valid if
$\sigma(a,b)\leq x\leq b$.

If we write (\ref{e3.2}) for $-f$ instead of $f$, we obtain
 $$ \mathcal{L}_{w,c}(f)(x) \geq
(1-c)(x-\sigma(a,b))\gamma+(\gamma-\Gamma)\nu(x,t^*). $$
 \end{proof}
 \begin{cor}
 For all $x\in[a,b]$ and $f\in C^{1}[a,b]$, we obtain
 \begin{equation}\label{e3.4}
\left| \mathcal{L}_{w}(f)(x) \right|\leq
\displaystyle\frac{\Gamma-\gamma}{m(a,b)}\int_{x}^{t^*}(t-x)w(t)dt.
\end{equation}
 \end{cor}
 \begin{proof}The inequality (\ref{e3.4}) follows from (\ref{e3.3})
 with $c=1$.
 \end{proof}
\begin{rem}The coefficient  $\displaystyle\frac{1}{m(a,b)}\int_{x}^{t^*}(t-x)w(t)dt$ is sharp in the sense that it cannot be replaced by a smaller one. The inequality (\ref{e3.4}) holds for all $x\in[a,b]$. Let $x=\sigma(a,b)$ and
$$ f(t)=\left\{\begin{array}{l} \Gamma(t-a),\, a\leq t<\sigma(a,b)\\
\vspace{-0.4cm}\\
\Gamma\left(\sigma(a,b)-a\right)+\gamma\left(t-\sigma(a,b)\right),\, \sigma(a,b)\leq t\leq b.\end{array}\right. $$
If $x=\sigma(a,b)$, then $t^*=b$ and it follows
$$  \mathcal{L}_{w}(f)(\sigma(a,b))=\displaystyle\frac{\Gamma-\gamma}{m(a,b)}\int_{\sigma(a,b)}^b\left(t-\sigma(a,b)\right)w(t)dt.$$
Therefore, in (\ref{e3.4})  equality holds.
\end{rem}
 \begin{cor}
 For all $x\in[a,b]$ and $f\in C^{1}[a,b]$, the following inequality
 holds
 \begin{align}
 &-\displaystyle\frac{\|f^{\prime}\|_{\infty}}{m(a,b)}\int_a^b|t-x|w(t)dt\leq(x-\sigma(a,b))\gamma+(\gamma-\Gamma)\nu(x,t^*)\label{e9}\\
 &\leq f(x)-\frac{1}{m(a,b)}\int_{a}^{b}f(t)w(t)dt\nonumber\\
 &\leq
 (x-\sigma(a,b))\Gamma+(\Gamma-\gamma)\nu(x,t^*)\leq \displaystyle\frac{\|f^{\prime}\|_{\infty}}{m(a,b)}\int_a^b|t-x|w(t)dt.\nonumber
 \end{align}
 \end{cor}
 \begin{proof}
 This result is a consequence of (\ref{e3.3}), for $c=0$.
 \end{proof}
 \begin{rem} a) Inequality (\ref{e3.4}) was established by J.
 Roumeliotis \cite{R}.

 b) Inequality (\ref{e9}) improves the classical weighted Ostrowski
 inequality
 $$ \left| f(x)-\frac{1}{m(a,b)}\displaystyle\int_{a}^{b}f(t)w(t)dt\right|\leq \|f^{\prime}\|_{\infty}\cdot\frac{1}{m(a,b)}\int_{a}^{b}|t-x|w(t)dt. $$

 \end{rem}

\section{Ostrowski-Gr\"{u}ss-type inequalities in terms of the least concave majorant}
The aim of this section is to extend the inequalities mentioned in the previous section, by using the least concave majorant of the modulus of continuity. This approach was inspired by a paper of Gavrea $\&$ Gavrea \cite{G} who were the first to observe the possibility of using moduli in this context.
\begin{pr}\label{p3.1} The linear functional $ \mathcal{L}_{w,c}:C[a,b]\to
\mathbb{R}$ satisfies

i) $\left| \mathcal{L}_{w,c}(f)(x) \right|\leq4\| f\|_{\infty}$, for
all $f\in C[a,b]$.

ii) $\left| \mathcal{L}_{w,c}(f)(x)
\right|\leq\left[(c-1)\left|x-\sigma(a,b)\right|+2\nu(x,t^*)\right]\|
f^{\prime}\|_{\infty}$, for all $f\in C^1[a,b]$, where $c\in[0,1]$
and $\nu$ is defined in Theorem \ref{t2.1}.
\end{pr}
\begin{proof} Inequality i) follows immediately from definition
 of $\mathcal{L}_{w,c}$. The second inequality is obtained after
 elementary calculations as follows:
 $$  \left| \mathcal{L}_{w,c}(f)(x)
\right|\leq
\displaystyle\int_{a}^b\left|{\mathcal{P}}(x,t)\right|f^{\prime}(t)dt\leq
\|f^{\prime}\|_{\infty}\int_{a}^b\left|{\mathcal{P}}(x,t)\right|dt
$$
If $a\leq x\leq \sigma(a,b)$, then
\begin{align}\label{xx}
\displaystyle \int_{a}^b\left|{\mathcal{P}}(x,t)\right|dt
&=\displaystyle\int_a^x {\mathcal{P}}(x,t)dt+\int_{t^*}^b
{\mathcal{P}}(x,t)dt-\int_x^{t^*} {\mathcal{P}}(x,t)dt\\
&=(1-c)(x-\sigma(a,b))+2\nu(x,t^{*}).\nonumber
\end{align}
By  similar reasoning, for
$\sigma(a,b)<x\leq b$,  it follows
\begin{align*}
\displaystyle \int_{a}^b\left|{\mathcal{P}}(x,t)\right|dt
&=\displaystyle-\int_a^{t^*} {\mathcal{P}}(x,t)dt-\int_{x}^b
{\mathcal{P}}(x,t)dt+\int_{t^*}^x {\mathcal{P}}(x,t)dt\\
&=(c-1)(x-\sigma(a,b))+2\nu(x,t^{*}).\nonumber
\end{align*}
\end{proof}
\begin{thm} If $f\in C[a,b]$, $c\in[0,1]$, then
$$  \left| \mathcal{L}_{w,c}(f)(x)
\right|\leq
2\tilde{\omega}\left(f;\displaystyle\frac{1}{2}\left[(c-1)\left|x-\sigma(a,b)\right|+2\nu(x,t^*)\right]\right),$$
where $\nu$ is defined in Theorem \ref{t2.1} and $\tilde{\omega}$ is
the least concave majorant of the usual modulus of continuity.
\end{thm}
\begin{proof} Taking an arbitrary $g\in C^1[a,b]$ and using Proposition
\ref{p3.1} we obtain
\begin{align*}
\left| \mathcal{L}_{w,c}(f)(x) \right|&\leq \left|
\mathcal{L}_{w,c}(f-g)(x) \right|+\left| \mathcal{L}_{w,c}(g)(x)
\right|\\
&\leq 4\|
f-g\|_{\infty}+\left[(c-1)\left|x-\sigma(a,b)\right|+2\nu(x,t^*)\right]\|g^{\prime}\|_{\infty}.
\end{align*}
Passing to the inf we arrive at
\begin{align*}
\left| \mathcal{L}_{w,c}(f)(x) \right|
&\leq 4\displaystyle\inf_{g\in C^1[a,b]}\left\{\|
f-g\|_{\infty}+\frac{1}{4}\left[(c-1)\left|x-\sigma(a,b)\right|+2\nu(x,t^*)\right]\|g^{\prime}\|_{\infty}\right\}\\
&=2\tilde{\omega}\left(f;\frac{1}{2}\left[(c-1)\left|x-\sigma(a,b)\right|+2\nu(x,t^*)\right]\right),
\end{align*}
so the result follows as a consequence of the relation \cite{10'}:
$$ \displaystyle\inf_{g\in C^1([a,b])}\left(\| f-g\|_{\infty}+\frac{t}{2}\|g^{\prime}\|\right)=\frac{1}{2}\tilde{\omega}(f;t), t\geq 0. $$
\end{proof}

 \begin{cor}
 For all $x\in[a,b]$ and $f\in C^{1}[a,b]$, we obtain
 \begin{equation*}
\left| \mathcal{L}_{w}(f)(x) \right|\leq
2\displaystyle\tilde{\omega}\left(f;\frac{2}{m(a,b)}\int_{x}^{t^*}(t-x)w(t)dt\right).
\end{equation*}
\end{cor}

 \begin{cor}
 For all $x\in[a,b]$ and $f\in C^{1}[a,b]$, the following inequality
 holds
 $$
  \left| f(x)-\frac{1}{m(a,b)}\int_{a}^{b}f(t)w(t)dt\right|=
  2\tilde{\omega}\left(f;\displaystyle\frac{1}{2m(a,b)}\left(\int_x^a(t-x)w(t)dt+\int_x^b(t-x)w(t)dt\right)\right).$$
  \end{cor}
  \section{Numerical example}
  In this section the inequality (\ref{e9}) is evaluated for some
  specific weight functions.

  \noindent{\bf 1.} Let the weight function $w$  be the
  probability density function of the Beta distribution,
  $$w_{p,q}(x)=\left\{\begin{array}{l}\displaystyle\frac{1}{B(p,q)}x^{p-1}(1-x)^{q-1},\,x\in[0,1],\\
  \vspace{-0.4cm}\\
  0,\, x\in \mathbb{R}\setminus[0,1],\end{array}\right.   $$
  where $B(p,q)=\displaystyle\int_{0}^1 x^{p-1}(1-x)^{q-1}dx$, $p,q>0$.

  Substituting $w_{p,q}$ in relation (\ref{e9}), it follows
  \begin{equation}\label{XY}\displaystyle\left(x-\frac{p}{p+q}\right)\gamma+(\gamma-\Gamma)\tilde{\nu}(x)\leq f(x)-\int_{a}^{b}w_{p,q}(t)f(t)dt\leq
  \left(x-\frac{p}{p+q}\right)\Gamma+(\Gamma-\gamma)\tilde{\nu}(x),
  \end{equation}
  where
  $$\tilde{\nu}(x)=\left\{\begin{array}{l} \displaystyle\frac{p}{p+q}-x-B(x;p+1,q)+xB(x;p,q),\, 0\leq x\leq \frac{p}{p+q},\\
  \vspace{-0.4cm}\\
  xB(x;p,q)-B(x;p+1,q),\,\displaystyle\frac{p}{p+q}<x\leq 1,\end{array}\right.  $$
  and $B(x;p,q)=\displaystyle\frac{1}{B(p,q)}\displaystyle\int_{0}^{x}t^{p-1}(1-t)^{q-1}dt$,
   $0\leq x\leq 1$ is the incomplete Beta function.

   In the below table , for $p=q=\displaystyle\frac{1}{2}$ and
   $f(t)=\displaystyle\frac{t^2}{2}, t\in[0,1]$ we calculate the left hand side
   and the right hand side of  inequality (\ref{XY}):
   \newpage
      \begin{center}
  {\it Table 1. Error estimate of $f(x)-\displaystyle\int_{0}^1w_{p,q}(t)f(t)dt$
  }

{\small \begin{tabular}{|l|l|l|l|l|}\hline
 $x$& $\tilde{\nu}(x)$& l.h.s of (\ref{XY})& r.h.s of (\ref{XY})& $f(x)-\int_{0}^1w_{p,q}(t)f(t)dt$ \\ \hline
0&0.500000000000000 &-0.500000000000000 &0 & -0.187500000000000\\
0.1&0.406636443481054 &-0.406636443481054 & 0.006636443481054&-0.182500000000000\\
0.2& 0.318514120706339&-0.318514120706339 &0.018514120706339 & -0.167500000000000\\
0.3&0.233428745882118 &-0.233428745882118 &0.033428745882118 & -0.142500000000000\\
0.4& 0.150335250602855&-0.150335250602855 &0.050335250602855 &-0.107500000000000 \\
0.5&0.068309886183791 & -0.068309886183791&0.068309886183791 &-0.062500000000000\\
0.6&0.086241033753880 & -0.086241033753880&0.186241033753880 &-0.007500000000000 \\
0.7&0.102438865447664 & -0.102438865447664&0.302438865447664 & 0.057500000000000\\
0.8&0.113681356007205 &-0.113681356007205 &0.413681356007205 & 0.132500000000000\\
0.9& 0.111469208180188&-0.111469208180188 & 0.511469208180188& 0.217500000000000\\
1& 0&0 & 0.500000000000000& 0.312500000000000\\
\hline
  \end{tabular}}
  \end{center}
$  $

  \noindent{\bf 2.} Let the weight $w$  be the probability density
  function of the normal distribution
  $$ w_{m,\sigma}(x)=\displaystyle\frac{1}{\sigma\sqrt{2\pi}}e^{-\frac{(x-m)^2}{2\sigma^2}},\,m,\sigma\in \mathbb{R}, \sigma>0,x\in\mathbb{R}. $$
  Then  we have
  $$ m(a,b)=F(b)-F(a),\textrm{where F is the cumulative distribution}, $$
  $$ \sigma(a,b)=\displaystyle{m}-\frac{\sigma}{\sqrt{2\pi}}\cdot\frac{e^{-\frac{(b-m)^2}{2\sigma^2}}-e^{-\frac{(a-m)^2}{2\sigma^2}}}{F(b)-F(a)}, $$
  $$\nu(x,t^*)\!=\!\left\{\begin{array}{l}\displaystyle\frac{1}{F(b)\!-\!F(a)}
  \left[\frac{-\sigma}{\sqrt{2\pi}}\left(e^{-\frac{(b\!-\!m)^2}{2\sigma^2}}\!-\!e^{-\frac{(x\!-\!m)^2}{2\sigma^2}}\right)\!+\!\left(m-\!x\right)\left(F(b)\!-\!F(x)\right)\right],
  a\!\leq\! x\!\leq \!\sigma(a,b),\\
  \vspace{-0.4cm}\\
  \displaystyle\frac{1}{F(b)\!-\!F(a)}
  \left[\frac{-\sigma}{\sqrt{2\pi}}\left(e^{-\frac{(a\!-\!m)^2}{2\sigma^2}}\!-\!e^{-\frac{(x\!-\!m)^2}{2\sigma^2}}\right)\!+\!\left(m-\!x\right)\left(F(a)\!-\!F(x)\right)\right],
   \sigma(a,b)\!<\! x\!\leq \!b,
  \end{array}\right.  $$
  If we consider the probability density function of the standard
  normal distribution, namely
  $$ w_{0,1}(x)=\displaystyle\frac{1}{\sqrt{2\pi}}e^{-\frac{x^2}{2}}, $$
  inequality (\ref{e9}) on the interval $[0,1]$  becomes
  $$\displaystyle\left(x-\frac{1-e^{-\frac{1}{2}}}{\phi(1)\sqrt{2\pi}}\right)\gamma+(\gamma-\Gamma)\tilde{\nu}(x)\leq
  f(x)-\frac{1}{\phi(1)}\int_a^b f(t)w_{0,1}(t)dt\leq \left(x-\frac{1-e^{-\frac{1}{2}}}{\phi(1)\sqrt{2\pi}}\right)\Gamma+(\Gamma-\gamma)\tilde{\nu}(x) , $$
  where
  $$ \tilde{\nu}(x)=\left\{\begin{array}{l}\displaystyle\frac{1}{\phi(1)}\left[\frac{-1}{\sqrt{2\pi}}\left(e^{-\frac{1}{2}}
  -e^{-\frac{x^2}{2}}\right)-x\left(\phi(1)-\phi(x)\right)\right], 0\leq x\leq \frac{1-e^{-\frac{1}{2}}}{\phi(1)\sqrt{2\pi}},\\
  \vspace{-0.4cm}\\
  \displaystyle\frac{1}{\phi(1)}\left[\frac{-1}{\sqrt{2\pi}}\left(1-
  e^{-\frac{x^2}{2}}\right)+x\phi(x)\right],  \frac{1-e^{-\frac{1}{2}}}{\phi(1)\sqrt{2\pi}}<x\leq 1.\end{array}
  \right. $$
 Here  $\phi$ is Laplace's function and $\phi(1)=0.3413$.

\bigskip
\noindent
 $\begin{array}{ll}
\textrm{\bf Ana Maria Acu}\\
\textrm{Lucian Blaga University of Sibiu} \\
 \textrm{Department of Mathematics and Informatics} \\
 \textrm{Str. Dr. I. Ratiu, No.5-7}  \\
 \textrm{RO-550012  Sibiu, Romania} \\
\textrm{e-mail: acuana77@yahoo.com}
\end{array} $

\bigskip
\noindent
 $\begin{array}{ll}
\textrm{\bf Heiner Gonska}\\
\textrm{University of Duisburg-Essen}\\
\textrm{Faculty of Mathematics}\\
\textrm{Forsthausweg 2}\\
\textrm{47057 Duisburg, Germany}\\
\textrm{e-mail: heiner.gonska@uni-due.de}
\end{array}$

\end{document}